\documentclass[12pt, a4paper, leqno]{amsart}
\usepackage[utf8]{inputenc}
\usepackage{amsmath}
\usepackage{amsfonts}
\usepackage{amssymb}
\usepackage{times}
\usepackage[T1]{fontenc} 
\usepackage{url} 
\usepackage{color,esint}
\usepackage[dvipsnames]{xcolor}
\usepackage[colorlinks, allcolors=RedViolet,pdfstartview=,pdfpagemode=UseNone]{hyperref} 
\usepackage{graphicx} 
\usepackage{enumitem}
\usepackage{tabularx}
 \usepackage{mathtools}
\usepackage{tikz}
\usepackage{mathrsfs}
\usetikzlibrary{arrows}

\let\oldmarginpar\marginpar
\renewcommand\marginpar[1]{\-\oldmarginpar[\raggedleft\footnotesize #1]%
{\raggedright\footnotesize #1}}

\usepackage{amsthm}
\theoremstyle{plain}

\newtheorem{lem}[equation]{Lemma}
\newtheorem{prop}[equation]{Proposition}

\theoremstyle{definition}
\newtheorem{defn}[equation]{Definition}

\newtheorem{example}[equation]{Example}

\theoremstyle{remark}

\numberwithin{equation}{section}

\newcommand{\R}{\mathbb{R}}

\newcommand{\Rn}{\mathbb{R}^n}

\renewcommand{\phi}{\varphi}
\renewcommand{\epsilon}{\varepsilon}
\def\le{\leqslant}
\def\leq{\leqslant}
\def\ge{\geqslant}

\def\phi{\varphi}
\def\rho{\varrho}
\def\vartheta{\theta}

\newcommand{\Phiw}{\Phi_{\text{\rm w}}}

\newcommand{\Phis}{\Phi_{\text{\rm s}}}

\def\supp{\operatorname{supp}}
\def\esssup{\operatornamewithlimits{ess\,sup}}

\def\dist{\qopname\relax o{dist}}

\def\essinf{\operatornamewithlimits{ess\,inf}}

\def\supp{\operatornamewithlimits{supp}}

\newcommand{\ainc}[1]{\hyperref[def:aInc]{{\normalfont(aInc){\ensuremath{_{#1}}}}}}
\newcommand{\adec}[1]{\hyperref[def:aDec]{{\normalfont(aDec){\ensuremath{_{#1}}}}}}
\newcommand{\inc}[1]{\hyperref[def:inc]{{\normalfont(Inc){\ensuremath{_{#1}}}}}}
\newcommand{\dec}[1]{\hyperref[def:dec]{{\normalfont(Dec){\ensuremath{_{#1}}}}}}
\newcommand{\azero}{\hyperref[def:A0]{{\normalfont(A0)}}}
\newcommand{\aone}{\hyperref[def:A1]{{\normalfont(A1)}}}
\newcommand{\atwo}{\hyperref[def:A2]{{\normalfont(A2)}}}

\newcommand{\aonen}[1]{\hyperref[aonen]{{\normalfont(A1-\ensuremath{#1})}}}
\newcommand{\VAn}[1]{\hyperref[VA1n]{{\normalfont(VA1-\ensuremath{#1})}}}
\newcommand{\wVAn}[1]{\hyperref[wVA1n]{{\normalfont(wVA1-\ensuremath{#1})}}}
\newcommand{\condMn}[1]{\hyperref[condM]{{\normalfont(M-\ensuremath{#1})}}}

\date{\today}

\begin{document}
\hypersetup{pageanchor=true}

\title[A revised condition for generalized Orlicz spaces on unbounded 
domains]{A revised condition for harmonic analysis in generalized Orlicz spaces on unbounded 
domains}

\author{Petteri Harjulehto}
\address{Petteri Harjulehto,
Department of Mathematics and Statistics,
FI-00014 University of Hel\-sin\-ki, Finland}
\email{\texttt{petteri.harjulehto@helsinki.fi}}

\author{Peter Hästö}
\address{Peter Hästö, Department of Mathematics and Statistics,
FI-20014 University of Turku, Finland}
\email{\texttt{peter.hasto@utu.fi}}

\author{Artur S{\l}abuszewski}
\address{Artur S{\l}abuszewski, Department of Mathematics and Information Sciences, Warsaw University of Technology, Poland}
\email{\texttt{artur.slabuszewski.dokt@pw.edu.pl}}


\begin{abstract}
Conditions for harmonic analysis in generalized Orlicz spaces 
have been studied over the past decade. One approach involves the generalized inverse of 
so-called weak $\Phi$-functions. It featured prominently in the monograph 
\emph{Orlicz Spaces and Generalized Orlicz Spaces}
[P.\ Harjulehto and P.\ Hästö, Lecture Notes in Mathematics, vol.\ 2236, Springer, Cham, 2019].
While generally successful, the inverse function formulation of 
the decay condition (A2) in the monograph contains a flaw, which we explain and correct in this note. 
We also present some new results related to the conditions, 
including a more general result for the density of smooth functions. 
\end{abstract}
 
\maketitle


\section{Introduction}

Generalized Orlicz spaces, or Musielak--Orlicz spaces, are useful in the analysis of a variety 
of models, including the variable exponent and double phase cases that have been fashionable 
lately. Harmonic analysis in this 
context has been studied over the past 10 years. 
In the 2019 monograph \cite{HarH_book}, conditions for harmonic analysis in generalized Orlicz spaces 
inspired by \cite{BarCM18, DieHHR11, Has15, MaeMOS13a}
were developed and refined. One idea was to formulate the assumptions  
in terms of the generalized inverse of so-called weak $\Phi$-functions 
that are not necessarily bijections (cf.\ \cite{BenHHK21, ChlGSW21, HadSV_pp, Has23, HasO22, HurOS23, SavSY23} 
for variants and developments of the conditions). The rationale for this is to obtain better behavior 
with respect to taking limits and other modifications. This approach has proven its power in 
numerous papers, but it turns out that the inverse function formulation of 
the decay condition \atwo{} for unbounded domains in \cite{HarH_book} contains a flaw, 
which we explain and correct in this note. 
Additionally, related new results are contained in Section~\ref{sect:other}. In particular, we show 
in Proposition~\ref{prop:equivalent} that \atwo{} implies \azero{}. 

\subsection*{The mistake in the book}

We start with two versions of the condition \atwo{} from the monograph \cite{HarH_book}. 
For every $\sigma>0$ there should exist $\beta\in(0,1]$ and $h\in L^1(\Omega) \cap L^\infty (\Omega)$ 
such that the following hold for a.e.\ $x,y\in \Omega$:
\begin{equation} \label{eq:old-A2}
\beta \phi^{-1}(x,\tau) 
\le 
\phi^{-1}(y,\tau)
\quad\text{when } \tau\in [h(x)+h(y), \sigma]
\end{equation}
or
\begin{equation}\label{eq:phi-A2}
\phi(x,\beta t) 
\le 
\phi(y,t) + h(x) + h(y)
\quad\text{when } \phi(y,t)\in [0, \sigma].
\end{equation}
In these (and later) conditions we assume without loss of generality that $h\ge 0$. 
Condition \eqref{eq:phi-A2} is called \atwo{} in \cite{Has15} whereas \eqref{eq:old-A2} carries 
that name in \cite{HarH_book}. The idea was that this does not lead to confusion since the 
conditions were supposed to be equivalent. 
Indeed, in \cite[Lemma~4.2.5]{HarH_book} they are claimed to be 
equivalent, but this is not the case as the next example shows.

\begin{example}\label{eg:no-A2}
Let $B(0,1) \subseteq \mathbb{R}^{n}$ be the unit ball centered at the origin and $\Omega:= B(0,1) \setminus \{0\}$. We define $\phi:\Omega\times [0,\infty)\to [0,\infty]$ by
\[
\phi(x,t) = \frac{t^{2}}{|x|}.
\]
Since $\Omega$ is bounded, $h:= \sigma\chi_\Omega \in L^1(\Omega)$. Thus $[h(x)+h(y), \sigma]$ is empty for all $x, y \in \Omega$, and so \eqref{eq:old-A2} holds.
It remains to show that $\phi$ does not satisfy \eqref{eq:phi-A2}.

Suppose that the inequality from \eqref{eq:phi-A2} holds for some $\sigma>0$, $\beta\in (0,1]$ and $h\in L^1(\Omega) \cap L^\infty (\Omega)$. 
Fix a non-exceptional point $y\in \Omega$ and $t\in (0,\infty)$ such that $\phi(y,t) \in [0,\sigma]$, the 
inequality \eqref{eq:phi-A2} holds for a.e $x\in \Omega$ and $|h(y)| \le\|h\|_\infty$. Then 
\[
\frac{\beta^{2}t^{2}}{|x|} \le \frac{t^{2}}{|y|} + 2 \|h\|_\infty
\]
for almost every $x\in \Omega$. When $x\to 0$, this gives a contradiction. Thus \eqref{eq:phi-A2} does not hold.
\end{example} 

The problem with the proof of \cite[Lemma~4.2.5]{HarH_book}, which claimed that 
\eqref{eq:old-A2} and \eqref{eq:phi-A2} are equivalent, is that
the set $[h(x)+h(y), \sigma]$ can be empty and hence \eqref{eq:old-A2} cannot necessary be 
applied for $\tau'$ in the proof. 
Here is an excerpt from the proof of \cite[Lemma~4.2.5]{HarH_book} with 
the problematic part high-lighted.

\begin{quote}
Assume \eqref{eq:old-A2} and denote $\tau:=\phi^{-1}(x,t)$. By \cite[Lemma~2.3.3]{HarH_book}, 
$\phi(x,\tau)=t$, and it follows from \eqref{eq:old-A2} that 
\[
\beta_2 \tau 
\le 
\phi^{-1}(y,\phi(x,\tau))
\]
for almost every $x,y\in \Omega$ whenever $\phi(x,\tau)\in [h(x)+h(y), \sigma]$. 
Then we apply $\phi(y,\cdot)$ to both sides and use \cite[Lemma~2.3.3]{HarH_book} to 
obtain that 
\[
\phi(y,\beta_2 \tau)
\le 
\phi(x,\tau)
\]
for the same range. If, on the other hand, $\phi(x,\tau)\in [0, h(x)+h(y))$, then 
we can find $\tau'>\tau$ such that $\phi(x,\tau')=h(x)+h(y)$ since $\phi\in \Phis(\Omega)$. 
As $\phi(x,\cdot)$ is increasing, we obtain by \textbf{the previous case applied for $\tau'$ that} 
\[
\phi(y,\beta_2 \tau)
\le 
\phi(y,\beta_2 \tau')
\le 
\phi(x,\tau')
=
h(x)+h(y).
\]
\end{quote}

In the next section we solve the problem by introducing a corrected \atwo{}-condition 
and re-prove some lemmas. First we recall some notation.

\subsection*{Notation and terminology}

For functions $\phi,\psi:\Omega \times [0,\infty)\to[0,\infty]$, we write 
$\phi\approx \psi$ or $\phi\simeq \psi$ if there exists $c\ge 1$ such that 
\[
\tfrac1c \phi(x,t)\le \psi(x,t)\le c \phi(x,t)
\quad\text{or}\quad
\phi(x,\tfrac tc)\le \psi(x,t)\le \phi(x,ct),
\] 
respectively, for a.e.\ $x \in \Omega$ and every $t\in [0,\infty)$.
A function $f:[0,\infty)\to [0,\infty]$ is \textit{almost increasing} (more precisely, $a$-almost increasing) if there
exists $a \ge 1$ such that $f(s) \le a f(t)$ for all $s \le t$.
\textit{Almost decreasing} is defined analogously.
By \textit{increasing} we mean that the inequality holds for $a=1$ 
(some call this non-decreasing), similarly for \textit{decreasing}. 

\begin{defn}
\label{def2-1}
We say that $\phi: \Omega\times [0, \infty) \to [0, \infty]$ is a 
\textit{weak $\Phi$-function}, and write $\phi \in \Phiw(\Omega)$, if 
the following conditions hold for a.e.\ $x \in \Omega$:
\begin{itemize}
\item 
$\phi(\cdot, |f|)$ is measurable for every measurable $f:\Omega\to \R$. 
\item
$t \mapsto \phi(x, t)$ is increasing. 
\item 
$\displaystyle \phi(x, 0) = \lim_{t \to 0^+} \phi(x,t) =0$ and $\displaystyle \lim_{t \to \infty}\phi(x,t)=\infty$.
\item 
$t \mapsto \frac{\phi(x, t)}t$ is $a$-almost increasing on $(0,\infty)$ with 
constant $a\ge 1$ independent of $x$.
\end{itemize}
If $\phi(x,\cdot)$ is additionally a convex and in $C([0,\infty); [0,\infty])$ 
for a.e.\ $x\in\Omega$, then $\phi$ is a 
\textit{strong $\Phi$-function} and we write $\phi \in \Phis(\Omega)$. 
\end{defn}

The \textit{left-inverse} of $\phi\in \Phiw(\Omega)$ is defined as 
\[
\phi^{-1}(x,\tau) 
:= 
\inf\{t\ge 0 \mid \phi(x,t)\ge \tau\}
\]
and the \textit{conjugate} is defined as 
\[
\phi^*(x,t):=\sup \{st - \phi(x,s) \mid s\ge 0\}. 
\]


\section{The corrected decay condition (A2)}\label{sect:A2}

We introduce a new, slightly modified \atwo{}-condition, that solves the problem explained in the previous section.
This yields changes to Lemmas 4.2.2--4.2.5 in the book \cite{HarH_book} (cf.\ Table~\ref{table:correspondence}).
No modification is needed in the book after Lemma~4.2.5 since \atwo{} is mostly used as \eqref{eq:phi-A2} 
via Lemma~4.2.5, and in other places changing to the new \atwo{} is trivial. Also the extra assumption in Lemma~\ref{lem:A2-bounded-set} compared to \cite[Lemma~4.2.3]{HarH_book} 
does not matter since when Lemma~4.2.3 is used in the book also \azero{} and \aone{} are assumed. 

Apart from the elegance of defining all conditions, \azero{}, \aone{} and \atwo{}, in 
terms of inverse functions, a motivation is that some features are easier to handle in this 
formulation, in particular invariance and conjugate functions. Since the invariance of 
\eqref{eq:phi-A2} under conjugation was shown in \cite[Lemma~4.2.4]{HarH_book} using \eqref{eq:old-A2}, that proof 
breaks down. Therefore, we need to establish $\phi^*$ satisfies \eqref{eq:phi-A2} if $\phi$ does 
using our new \atwo{}.

\begin{table}[ht!]
\caption {Correspondence between results in the book and this note. }\label{table:correspondence}
\begin{tabular}{lll}
Book & This note & Change\\ 
\hline
Definition 4.2.1 & Definition~\ref{def:As}(A2) & definition modified\\ 
Lemma 4.2.2 & Lemma~\ref{lem:A2-invariant} & old proof works\\
Lemma 4.2.3 & Lemma~\ref{lem:A2-bounded-set} & stronger assumption and new proof\\ 
Lemma 4.2.4 & Lemma~\ref{lem:phi*_A2} & auxiliary lemma and old proof\\ 
Lemma 4.2.5 & Lemma~\ref{lem:A2original} &new proof\\
\end{tabular}
\end{table}

Our corrected \atwo{}-condition with the inverse $\phi^{-1}$ is included in the following definition. 
Compared to \eqref{eq:old-A2}, it is more directly related to \eqref{eq:phi-A2}, but not as elegant.

\begin{defn}\label{def:As}
For $\phi:\Omega\times [0,\infty)\to [0,\infty]$ and $p,q>0$ we define some conditions.
\begin{itemize}[leftmargin=4em]
\item[(aInc)$_p$] \label{def:aInc} 
There exists $a_p\ge 1$ such that $t \mapsto \frac{\phi(x,t)}{t^{p}}$ is $a_p$-almost 
increasing in $(0,\infty)$ for a.e.\ $x\in\Omega$.
\item[(aDec)$_q$] \label{def:aDec}
There exists $a_q\ge 1$ such that $t \mapsto \frac{\phi(x,t)}{t^{q}}$ is $a_q$-almost 
decreasing in $(0,\infty)$ for a.e.\ $x\in\Omega$.
\item[(A0)]\label{def:A0}
There exists $\beta \in(0, 1]$ such that $\beta \le \phi^{-1}(x,1) \le \frac1\beta$ 
for a.e.\ $x \in \Omega$. 
\item[(A1)]\label{def:A1}
There exists $\beta \in (0,1]$ such that, for a.e.\ $x,y\in \Omega\cap B$,
\[ 
\beta \phi^{-1}(x, \tau) \le \phi^{-1}(y,\tau) \quad\text{when}\quad \tau \in \bigg[1, \frac{1}{|B|}\bigg].
\]
\item[(A2)]\label{def:A2}
For every $\sigma>0$ there exist $\beta\in(0,1]$ and $h\in L^1(\Omega) \cap L^\infty (\Omega)$, $h\ge 0$,
such that, for a.e.\ $x,y\in \Omega$, 
\[
\beta \phi^{-1}(x,\tau) 
\le 
\phi^{-1}(y,\tau+h(x)+h(y))
\quad\text{when}\quad \tau\in [0, \sigma].
\]
\end{itemize} 
\end{defn}

Since $\phi\simeq\psi$ if and only if $\phi^{-1}\approx \psi^{-1}$ by
Theorem~2.3.6 of \cite{HarH_book}, we see that $\phi$ satisfies \atwo{} if and only if $\psi$ does. 
This proves the next result. 

\begin{lem}\label{lem:A2-invariant}
\atwo{} is invariant under equivalence of weak $\Phi$-functions.
\end{lem}

Like \eqref{eq:old-A2}, the new \atwo{}-condition only concerns the behavior of $\phi$ at infinity, as 
the following result shows. The next lemma contains an extra assumption, \azero{}, compared to the book. The fact that \eqref{eq:old-A2} holds in all bounded domains even when $\phi$ is not 
locally bounded (cf.\ Example~\ref{eg:no-A2}) suggests that the old condition is not reasonable.

\begin{lem}\label{lem:A2-bounded-set}
Let $\Omega \subset \Rn$ be bounded.
If $\phi \in \Phiw (\Omega)$ satisfies \azero{}, then it satisfies \atwo{}. 
\end{lem}

\begin{proof}
Let $\sigma>0$. 
Since $\phi$ satisfies \ainc{1}, $\phi^{-1}$ satisfies \adec{1} by Proposition 2.3.7 of \cite{HarH_book}. 
If $\tau \in [0, 1]$, then $\phi^{-1}(x,\tau) \le \phi^{-1}(x, 1) \le \frac1{\beta}$ by \azero{}.
If $\tau \in (1, \sigma]$, then by \adec{1} and \azero{},
\[
\phi^{-1}(x,\tau) \le a \tau \phi^{-1}(x, 1) \le a \sigma \tfrac1{\beta}.
\]
Choose $h := \chi_\Omega$. Since $h$ and $\Omega$ are bounded, $h \in L^1(\Omega)\cap L^\infty(\Omega)$. 
Furthermore, 
\[
\phi^{-1}(y,\tau+h(x)+h(y)) \ge \phi^{-1}(y,1) \ge \beta
\]
by \azero{}. Combining the inequalities, we obtain the \atwo{}-inequality 
\[
\frac{\beta^2}{\max\{1,a\sigma\}} \phi^{-1}(x,\tau) \le \phi^{-1}(y,\tau+h(x)+h(y))
\quad\text{for every }\tau \in [0, \sigma]. \qedhere
\]
\end{proof}

The inverse is also useful when dealing with the conjugate function. 
The following proof is as in \cite{HarH_book}, once we use an auxiliary result proved below. 

\begin{lem}\label{lem:phi*_A2}
If $\phi\in\Phiw(\Omega)$ satisfies \atwo{}, then so does $\phi^*$.
\end{lem}
\begin{proof}
Let us write $\tilde\tau := \max\{\tau,h(x)+h(y)\}$ and use 
Proposition~\ref{prop:equivalent}(3). Thus it suffices to assume
$\beta \phi^{-1}(x,\tilde\tau) 
\le 
\phi^{-1}(y,\tilde\tau)$ and show the same inequality for $\phi^*$.
We multiply the inequality by $(\phi^*)^{-1}(x,\tilde\tau)(\phi^*)^{-1}(y,\tilde\tau)$
and use $\frac1c\tilde\tau \le (\phi^*)^{-1}(x,\tilde\tau)\phi^{-1}(x,\tilde\tau) \le c\tilde\tau$ \cite[Theorem~2.4.8]{HarH_book}
to obtain that 
\[
\tfrac\beta {c^2} \tilde\tau (\phi^*)^{-1}(y,\tilde\tau)
\le 
\tilde\tau (\phi^*)^{-1}(x,\tilde\tau).
\]
For $\tilde\tau>0$, we divide the previous inequality by $\tilde\tau$ to get \atwo{} of $\phi^*$; for 
$\tilde\tau=0$, \atwo{} is trivial, since $(\phi^*)^{-1}(x, 0)=0$. 
\end{proof}

Next we prove the equivalence of \atwo{} and \eqref{eq:phi-A2}. Note that the 
equivalence is an easy exercise if $\phi$ is a strictly increasing bijection, 
but here we consider a more general case.

\begin{lem}\label{lem:A2original}
The weak $\Phi$-function $\phi\in \Phiw(\Omega)$ satisfies \atwo{} if and only if 
it satisfies \eqref{eq:phi-A2}.
\end{lem}

\begin{proof}
We first show that \eqref{eq:phi-A2} is invariant under equivalence of 
$\Phi$-functions. Let 
$\phi \simeq \psi \in \Phiw(\Omega)$ with constant $L\ge 1$. By \eqref{eq:phi-A2} of $\phi$,  
\[
\psi(x, \tfrac \beta L t) 
\le \phi(x,\beta t) 
\le \phi(y,t) + h(x) + h(y)
\le \psi(y,Lt) + h(x) + h(y),
\] 
for $\phi(y,t) \in [0,\sigma]$. Denote $t':=Lt$. If $\psi(y,t') \in [0,\sigma]$, then 
$\phi(y,t) \in [0,\sigma]$ and the previous inequality gives 
$\psi(x, \tfrac \beta {L^2} t') \le \psi(y,t') + h(x) + h(y)$, which is \eqref{eq:phi-A2} for $\psi$. 
By Lemma~\ref{lem:A2-invariant}, \atwo{} is invariant under equivalence of weak $\Phi$-functions. 
By Theorem~2.5.10 of \cite{HarH_book}, every weak $\Phi$-function is equivalent to a strong 
$\Phi$-function. Hence it suffices to show the claim for $\phi\in\Phis(\Omega)$, so the 
inverse is better behaved since $\phi$ is a surjection onto $[0, \infty)$. Specifically, 
$\phi(x,\phi^{-1}(x,\tau))=\tau$ for every $\tau \ge 0$ and $\phi^{-1}(x,\phi(x,t))=t$ 
provided $\phi(x,t)\in (0,\infty)$ \cite[Lemma~2.3.3 and Corollary~2.3.4]{HarH_book}. 



Assume first that \atwo{} holds and $\phi(y,t)\in [0, \sigma]$. Let $\tau:=\phi(y,t)$. 
If $\tau\in (0,\sigma]$, then 
\begin{align*}
\beta \phi^{-1}(y,\tau) \le \phi^{-1}(x, \tau+h(x)+h(y))
&\quad\Leftrightarrow\quad
\beta t \le \phi^{-1}\big(x, \phi(y,t)+h(x)+h(y)\big) \\
&\quad\Rightarrow\quad
\phi(x, \beta t) \le \phi\big(x,\phi^{-1}\big(x, \phi(y,t)+h(x)+h(y)\big)\big)\\
&\quad\Leftrightarrow\quad
\phi(x, \beta t) \le \phi(y,t)+h(x)+h(y)
\end{align*}
and so \eqref{eq:phi-A2} holds. 
When $\tau=0$, we denote $t_0:=\max\{s\ge 0 : \phi(y, s)=0\}$, so that $t\in [0, t_0]$. 
Let $t_i\searrow t_0$. 
Then the previous argument applies to $\tau_i:=\phi(y, t_i)\in (0,\sigma]$ and yields
\[
\phi(x, \beta t) 
\le 
\phi(x, \beta t_i) 
\le \phi(y,t_i)+h(x)+h(y)
\to
h(x)+h(y),
\]
where we used that $\phi$ is continuous. Also in this case \eqref{eq:phi-A2} holds.

Assume conversely that \eqref{eq:phi-A2} holds and $\tau\in [0,\sigma]$. Let $t:=\phi^{-1}(y,\tau)$ so that 
$\phi(y,t)=\tau$. 
If $\phi(x,\beta t)>0$, then 
\begin{align*}
\phi(x, \beta t) \le \phi(y,t)+h(x)+h(y)
&\quad\Rightarrow\quad
\phi^{-1}(x,\phi(x, \beta t)) \le \phi^{-1}\big(x, \tau+h(x)+h(y)\big) \\
&\quad\Leftrightarrow\quad
\beta \phi^{-1}(y,\tau) \le \phi^{-1}(x, \tau+h(x)+h(y))
\end{align*}
If $\phi(x,\beta t)=0$ and $\tau = \phi(y,t)>0$, then 
\[
\beta \phi^{-1}(y, \tau) = \beta t
\le 
\max\{s\ge 0 \mid \phi(x, s)=0\}
\le
\phi^{-1}(x, \tau). 
\] 
Finally, if $\tau=0$, then \atwo{} automatically holds since the left-hand side in the 
\atwo{}-inequality equals zero. 
In all cases, \atwo{} follows from \eqref{eq:phi-A2}. 
\end{proof}


\section{Other results}\label{sect:other}

Next we collect conditions that are equivalent with \atwo{}. 
The starting point of this paper was the observation that 
\eqref{eq:old-A2} and \eqref{eq:phi-A2} are not equivalent. However, the following proposition 
shows that this can be remedied by some additional assumptions. 
We define one more auxiliary condition: for every $\sigma>0$, there exist $\beta\in(0,1]$ and $h\in L^1(\Omega) \cap L^\infty (\Omega)$ such that, for a.e.\ $x,y \in \Omega$, 
\begin{equation}\label{eq:max-A2}
\beta \phi^{-1}\big(x,\max\{ \tau,h(x)+h(y) \}\big) \le \phi^{-1}\big(y,\max\{ \tau,h(x)+h(y) \}\big)
\quad\text{when}\quad \tau\in [0, \sigma];
\end{equation}

\begin{prop}\label{prop:equivalent}
Let $\phi \in \Phiw(\Omega)$. The following are equivalent:
\begin{enumerate}
\item $\phi$ satisfies \atwo{}.
\item $\phi$ satisfies \eqref{eq:phi-A2}.
\item $\phi$ satisfies \eqref{eq:max-A2}.
\item $\phi$ satisfies \eqref{eq:old-A2} for $h$ with $\|h\|_\infty\leq\frac{\sigma}{2}$.
\item $\phi$ satisfies \eqref{eq:old-A2} and \azero{}. 
\end{enumerate}
\end{prop}
\begin{proof}
From Lemma~\ref{lem:A2original} we have equivalence between (1) and (2). 
The implication from (4) to (2) follows 
from the proof of the \cite[Lemma~4.2.5]{HarH_book}, since 
the additional assumption $\|h\|_\infty\leq\frac{\sigma}{2}$ ensures that 
$[h(x)+h(y),\sigma]$ is non-empty for a.e.\ $x,y\in\Omega$. 
We conclude the proof by showing that (1) $\Rightarrow$ (3) $\Rightarrow$ (4) and (4) $\Leftrightarrow$ (5).

First we show that (1) implies (3).
Let $\tilde\tau := \max\{\tau,h(x)+h(y)\}$. 
Since $\phi$ satisfies \ainc{1}, $\phi^{-1}$ satisfies \adec{1} by Proposition 2.3.7 of \cite{HarH_book}. 
Hence
\begin{equation*}
\phi^{-1}(y,\tau+h(x)+h(y)) 
\le \phi^{-1}(y,2 \tilde\tau) 
\le 2a \phi^{-1}(y,\tilde\tau) 
\le 2a \phi^{-1}(y,\tau+h(x)+h(y)).
\end{equation*}
Thus \atwo{} is equivalent with
\[
\beta \phi^{-1}(x,\tau) 
\le 
\phi^{-1}(y,\tilde\tau)
\]
for almost every $x,y\in \Omega$ and every $\tau\in [0, \sigma]$.
To obtain (3), we need to also show that $\beta \phi^{-1}(x,h(x)+h(y))\le \phi^{-1}(y,\tilde\tau)$ 
with possibly different $\beta>0$.
If $h(x)+h(y)\in (0,\sigma]$, then we can use the previous inequality with $\tau=h(x)+h(y)$ and obtain
\[
\beta \phi^{-1}(x,h(x)+h(y)) 
\le \phi^{-1}(y,h(x)+h(y)) 
\le
\phi^{-1}(y,\tilde\tau).
\]
If $h(x)+h(y)\in (\sigma,\infty)$, then we use 
\adec{1} for $\phi^{-1}$, 
and  \atwo{} with $\tau=\sigma$ to conclude
\[
\beta \phi^{-1}(x,h(x)+h(y)) 
\le\tfrac{a }{\sigma} (h(x)+ h(y)) \, \beta\phi^{-1}(x,\sigma)
\le\tfrac{2a}{\sigma} \|h\|_\infty\,\phi^{-1}(y,\sigma). 
\]
Since $\sigma\le h(x)+h(y)\le\tilde\tau$, (3) follows. 

Next we show that (3) implies (4). For $\sigma>0$, let 
$\beta\in (0,1]$ and $h\in L^1(\Omega) \cap L^\infty (\Omega)$, $h\ge 0$, be from \eqref{eq:max-A2}. 
If $\|h\|_\infty\leq \frac{\sigma}{2}$, there is nothing to prove, so we consider $\|h\|_\infty> \frac{\sigma}{2}$. 
We denote $\tilde{h}:= \tfrac{\sigma}{2 \|h\|_\infty}h$ and show that \eqref{eq:old-A2} holds with 
function $\tilde h$ and some constant $\tilde\beta$, which establishes (4). 
Let $\tau\in [\tilde{h}(x) + \tilde{h}(y), \sigma]$ 
so that $h(x)+ h(y) \le \tfrac{2 \|h\|_\infty}{\sigma}\tau$.
If $\tau\ge h(x) + h(y)$, then (4) follows from \eqref{eq:max-A2}. 
When $\tau< h(x) + h(y)$, we use that $\phi^{-1}$ is increasing, the inequality \eqref{eq:max-A2} 
and \adec{1} of $\phi^{-1}$ to obtain that 
\begin{align*}
\beta \phi^{-1}(x,\tau) 
\le \beta \phi^{-1}(x,h(x)+h(y)) 
\le \phi^{-1}(y,h(x)+ h(y)) 
\le \frac{2a \|h\|_\infty}{\sigma}\phi^{-1}(y,\tau). 
\end{align*}
Thus (4) holds with $\tilde\beta:=\frac{\beta\sigma}{2a\|h\|_\infty}$. 

It remains to show that (4) and (5) are equivalent. 
Assume first that (4) holds. 
Then \eqref{eq:old-A2} holds so we need only establish \azero{}. 
We use \eqref{eq:old-A2} with $\sigma=1$ and $\|h\|_{\infty}\le\frac{1}{2}$ for $\tau=1$ 
and find that 
\[
\beta \phi^{-1}(x,1) 
\le 
\phi^{-1}(y,1)
\]
for a.e.\ $x,y\in \Omega$. Thus
\begin{equation*}
\beta\esssup_{y\in \Omega} \phi^{-1}(y,1)\le\phi^{-1}(x,1) \le \frac{1}{\beta} \essinf_{y\in \Omega} \phi^{-1}(y,1)
\end{equation*}
 for a.e $x\in \Omega$. Since $\phi^{-1}(x,1) \in (0,\infty)$, this implies that
\begin{equation*}
0<\essinf_{y\in \Omega} \phi^{-1}(y,1) \le\esssup_{y\in \Omega} \phi^{-1}(y,1)<\infty.
\end{equation*}
Hence, \azero{} holds with constant
\begin{equation*}
\tilde{\beta} = \beta \min\Big\{\esssup_{y\in \Omega} \phi^{-1}(y,1), \big(\essinf_{y\in \Omega} \phi^{-1}(y,1)\big)^{-1}\Big\}.
\end{equation*}

Assume finally that (5) holds, let $\beta$ and $h$ be from \eqref{eq:old-A2} and define 
$\tilde h:=\min\{h, \frac \sigma 2\}$. We show that \eqref{eq:old-A2} holds with function $\tilde h$ (and 
different $\beta$), which gives (4). 
If $h(x), h(y)\le \frac \sigma 2$, then $\tilde h(x)=h(x)$ and $\tilde h(y)=h(y)$ and there is nothing to show. 
Otherwise $h(x)+h(y)\ge \frac \sigma 2$, so it suffices to prove the inequality from \eqref{eq:old-A2} 
for all $\tau\in [\frac \sigma 2, \sigma]$. For this we use \adec{1} and \azero{}:
\[
\phi^{-1}(x,\sigma) 
\le 
\max\{1, a \sigma\} \phi^{-1}(x,1) 
\le
\frac {\max\{1, a\sigma\}}{\beta^{2}} \phi^{-1}(y,1)
\le
\frac{\max\{1, a\sigma\} \max\{1, \tfrac {2a}\sigma\}}{\beta^{2} } \phi^{-1}(y,\tfrac \sigma2).
\]
Thus \eqref{eq:old-A2} holds with function $\tilde h$ and constant 
$\tilde\beta:=\frac{\beta^{2} }{\max\{1, a \sigma\} \max\{1, 2a/\sigma\}}$.
\end{proof}

In Lemma~\ref{lem:A2-bounded-set} we showed that \azero{} implies \atwo{} in 
bounded sets and in the previous proposition we saw that \atwo{} implies \azero{} in any set. 
The next result shows that \aone{} implies both of them.

\begin{prop}\label{prop:A2-bounded-set}
Let $\Omega \subset \Rn$ be bounded.
If $\phi \in \Phiw (\Omega)$ satisfies \aone{}, then it satisfies \azero{} and \atwo{}.
\end{prop}

\begin{proof}
Let $\phi$ satisfy \aone{}. 
By Lemma~\ref{lem:A2-bounded-set}, it suffices to show \azero{}. 
Since $\lim_{t \to 0^+} \phi(x, t) =0$ (by $\phi\in\Phiw(\Omega)$),
we obtain from the definition of inverse that $\phi^{-1}(x, 1) >0$ for a.e.\ $x \in \Omega$. 
Since $\Omega$ is precompact, we can cover it by
finitely many balls $B_1, \ldots, B_m$ of measure equal to $1$.
In each ball we choose $x_i \in \Omega \cap B_i$ outside the exceptional set so \aone{} implies that 
$\beta \phi^{-1}(x_i, 1) \le \phi^{-1}(y, 1) \le \frac1\beta\phi^{-1}(x_i, 1)$ for almost every $y\in B_i \cap \Omega$.
Since there are only finitely many balls, we obtain that
\[
0<\beta \min_{i\in \{1, \ldots, m\}} \phi^{-1}(x_i, 1) 
\le \phi^{-1}(y, 1) 
\le 
\frac1\beta \max_{i\in \{1, \ldots, m\}} \phi^{-1}(x_i, 1)<\infty
\]
for almost every $y \in \Omega$, and hence \azero{} holds.
\end{proof}

The previous result was inspired by Kami\'nska and \.{Z}yluk's recent study \cite{KamZ_pp23} 
of density of smooth functions in Sobolev spaces (see also \cite{BorC22}), 
where they show that \azero{} and \atwo{} are not needed for density, 
in contrast to earlier results like Theorem~6.4.4 from \cite{HarH_book}. 
Using Proposition~\ref{prop:A2-bounded-set}, we can easily 
obtain the theorem without assuming \azero{} and \atwo{}. 
Similar modifications can be done for the other results in \cite[Section~6.4]{HarH_book} 
and other results which are local in character.

\begin{prop}\label{prop:density}
Let $\phi\in\Phiw(\Rn)$ satisfy \aone{} and \adec{q} for some $q>1$. Then 
$C^\infty_0(\Rn)$ is dense in $W^{1,\phi}(\Rn)$. 
\end{prop}
\begin{proof}
Let $u \in W^{1,\phi}(\Rn)$ and $\epsilon \in (0,1)$.
By \cite[Lemma~6.4.1]{HarH_book}, we may assume that
$u$ has compact support in $\Rn$. Denote $\Omega:= \{ x\in \Rn \mid \dist(x, \supp u)<1\}$. 
Then $\Omega$ is bounded and by Proposition~\ref{prop:A2-bounded-set}, 
$\phi$ satisfies \azero{}, \aone{} and \atwo{} in $\Omega$. 
Let $\sigma_\epsilon$ be a standard mollifier. Then $u*\sigma_\epsilon$ belongs to
$C^\infty_0(\Omega)$ and 
\begin{align*}
\nabla(u * \sigma_\epsilon) - \nabla u = (\nabla u) * \sigma_ 
\epsilon - \nabla u.
\end{align*}
Thus the claim follows since $\|f * \sigma_\epsilon - f\|_{L^{\phi}(\Omega)} \to 0$ as $\epsilon \to 0^+$
for every $f \in L^{\phi}(\Omega)$ by \cite[Theorem~4.4.7]{HarH_book} and since the norm
$ \|u\|_{L^{\phi}(\Omega)} + \|\nabla u\|_{L^{\phi}(\Omega)}$ is equivalent with $\|u\|_{W^{1, \phi}(\Omega)}$ by \cite[Lemma~6.1.5]{HarH_book} (see \cite[Definitions~3.2.1 and 6.1.2]{HarH_book} for the definition of the norms).
\end{proof}

\end{document}